\documentclass[11pt
]{article}

\usepackage[all]{xy}

\usepackage[latin1]{inputenc}
\usepackage{amsfonts}
\usepackage{amsmath}
\usepackage{amssymb}
\usepackage{amscd}
\usepackage{amsthm}
\usepackage{stmaryrd}
\usepackage{indentfirst}
\usepackage[hmargin=2.6cm
,vmargin=3.9cm
]{geometry}

\usepackage{faktor}

\usepackage{epsfig}

\newtheorem{thm}{Theorem}[section]
\newtheorem{lemma}[thm]{Lemma}
\newtheorem{conj}[thm]{Conjecture}
\newtheorem{prop}[thm]{Proposition}
\newtheorem{coroll}[thm]{Corollary}

\theoremstyle{definition}

\newtheorem{defin}[thm]{Definition}
\newtheorem{rem}[thm]{Remark}
\newtheorem{exam}[thm]{Example}

\newtheorem{notation}[thm]{Notation}
\newtheorem*{acknow}{Acknowledgements}

\newcommand{\R}{{\mathbb{R}}}

\newcommand{\Z}{{\mathbb{Z}}}
\newcommand{\N}{{\mathbb{N}}}

\newcommand{\cI}{{\mathcal{I}}}

\newcommand{\cQ}{{\mathcal{Q}}}

\newcommand{\cZ}{{\mathcal{Z}}}

\newcommand{\fc}{{:\ }}

\newcommand{\wt}{\widetilde}
\newcommand{\wh}{\widehat}

\newcommand{\tb}{\textbf}

\DeclareMathOperator{\im}{im}

\DeclareMathOperator{\Vol}{Vol}

\DeclareMathOperator{\pr}{pr}

\DeclareMathOperator{\Ham}{Ham}

\DeclareMathOperator{\Int}{Int}

\DeclareMathOperator{\res}{res}

\begin{document}

\title{Hofer continuous quasi-morphisms on Liouville manifolds}

\author{Frol Zapolsky}


\setcounter{tocdepth}{3}

\renewcommand{\labelenumi}{(\roman{enumi})}

\maketitle


\begin{abstract}
   We show that, given a complete Liouville manifold, any homogeneous quasi-morphism on its Hamiltonian group, which satisfies a strengthened version of Hofer continuity called stability, must vanish. This partially addresses a conjecture due to L.\ Polterovich.
\end{abstract}

\section{Background and main result}

A function $\mu \fc G \to \R$, where $G$ is a group, is a quasi-morphism if
$$D(\mu):=\sup_{a,b\in G}|\mu(ab)-\mu(a) - \mu(b)| < \infty\,.$$
The quantity $D(\mu)$ is called the defect of $\mu$. We say that $\mu$ is homogeneous if in addition
$$\mu(a^n) = n\mu(a)\quad \text{for all }a\in G\,, n \in \Z\,.$$

Quasi-morphisms on infinite-dimensional groups of symmetries have been the subject of substantial interest in the last decades, see, for instance, the introduction to \cite{Borman_Zapolsky_QMs_Contacto} and the references therein for a sample. Here we are concerned with quasi-morphisms on the Hamiltonian group of a symplectic manifold, as well as its universal cover. These objects have been a fascinating theme, both in terms of their construction, as well as their applications. Existing constructions of such quasi-morphisms fall into two categories:
\begin{itemize}
  \item \emph{`Soft'} or \emph{geometric}: these utilize classical geometric and analytic tools, and Poincar\'e's rotation number appears as a central theme;
  \item \emph{`Hard'} or \emph{Floer-theoretic}: these ones use properties of solution spaces of elliptic PDEs in the guise of Floer theory, and are extremely involved.
\end{itemize}
Examples of soft constructions are given by Ruelle \cite{Ruelle_Rot_Numbers_Diffeos} for the unit disk in $\R^2$, Barge--Ghys \cite{Barge_Ghys_Cocycles_Euler_Maslov} for the unit ball in $\R^{2n}$, Gambaudo--Ghys \cite{Gambaudo_Ghys_Commutators_Diffeos_Surfaces} for surfaces, Entov \cite{Entov_Commutator_Length_Symplectos} for symplectic Calabi--Yau manifolds of finite volume, Py \cite{Py_QMs_Invariant_Calabi} for monotone symplectic manifolds of finite volume. This theme culminated with a beautiful general construction by Shelukhin \cite{Shelukhin_Action_Hom_QMs_Moment_Maps} for general symplectic manifolds of finite volume. The Calabi homomorphism \cite{Calabi_Group_Auto_Sympl_Mfd}, although not a `true' quasi-morphism, is central to this whole topic, including the present paper, since many homogeneous quasi-morphisms on the Hamiltonian group coincide it when restricted to sufficiently small open subsets. We refer the reader to the introduction of \cite{Shelukhin_Action_Hom_QMs_Moment_Maps} for a pleasant review of the subject of quasi-morphisms on the Hamiltonian group.

Floer-theoretic constructions originate with Entov--Polterovich \cite{Entov_Polterovich_Calabi_QM_QH}, with subsequent extensions and generalizations by Ostrover \cite{Ostrover_Calabi_QMs_Nonmonotone}, Entov--Polterovich \cite{Entov_Polterovich_Sympl_QSs_Semisimplicity_Quantum_Homology} (which also details an observation of McDuff regarding the role of semisimplicity of quantum homology), Caviedes-Castro \cite{Caviedes_Castro_Calabi_QMs_Monotone_Coad_Orbits}, Lanzat \cite{Lanzat_QMs_Sympl_QSs_Cx_Sympl_Mfds}. Monzner--Vichery--Zapolsky \cite{Monzner_Vichery_Zapolsky_Partial_QMs_QSs_Cot_Bundles} construct a function on the Hamiltonian group of a cotangent bundle; thanks to Shelukhin's result \cite{Shelukhin_Viterbo_Conj_Zoll_Symm_Spaces}, its restriction to the Hamiltonian group of each codisk bundle is a quasi-morphism.

Some of the above constructions in fact yield functions on the compactly supported Hamiltonian group (or its universal cover) of a complete Liouville manifold, namely Barge--Ghys ($\R^{2n}$), Entov (Liouville manifolds with $c_1=0$), Shelukhin (any Liouville manifold), and Monzner--Vichery--Zapolsky (cotangent bundles of symmetric Zoll spaces). However those functions only become quasi-morphisms when restricted to the Hamiltonian group of Liouville subdomains, with their defect growing with the size of the domain. By `size' here we mean volume for the first three constructions, and the radius of a codisk bundle for the last one. This naturally leads to the following
\begin{conj}
  The only homogeneous quasi-morphisms on the (universal cover of the) Hamiltonian group of a complete Liouville manifold are multiples of the Calabi homomorphism.
\end{conj}
\noindent This conjecture was formulated for cotangent bundles by L.\ Polterovich circa 2010 in a discussion of the Monzner--Vichery--Zapolsky construction. Of course, the more general form appearing here also makes sense. In this note we address the case of this conjecture for so-called stable quasi-morphisms. Here stability is meant in the sense of Entov--Polterovich--Zapolsky \cite{EPZ_QMs_Poisson_Bracket}, and we proceed to define it. In what follows, all symplectic manifolds are assumed to have no boundary. If $(M,\omega)$ is a symplectic manifold, $\Ham(M,\omega)$ will stand for its group of Hamiltonian diffeomorphisms generated by compactly supported Hamiltonians, and $\wt\Ham(M,\omega)$ for its universal cover. For $H \in C^\infty_c(M\times [0,1])$ we let $\phi_H^t$ be the time-$t$ Hamiltonian flow generated by $H$, denote $\phi_H=\phi_H^1$, and let $\wt\phi_H \in \wt\Ham(M,\omega)$ be the class of the isotopy $(\phi_H^t)_{t\in[0,1]}$.

A quasi-morphism $\mu \fc \wt\Ham(M,\omega) \to \R$ is called \emph{stable} if it satisfies
$$\int_0^1\min_M(G_t-H_t)\,dt \leq \mu(\wt\phi_G) - \mu(\wt\phi_H) \leq \int_0^1\max_M(G_t-H_t)\,dt$$
for all time-dependent Hamiltonians $G,H$. Stability is a strengthened variant of Lipschitz continuity with respect to Hofer's distance, and it is satisfied by all known versions of Floer-homological spectral invariants, and thus it is a natural assumption on potential quasi-morphisms on the Hamiltonian group. Here is then our result.
\begin{thm}\label{thm:main}
  Let $(W,\omega)$ be a connected complete Liouville manifold. Then any stable homogeneous quasi-morphism on $\wt \Ham(W,\omega)$ vanishes identically.
\end{thm}
\noindent Informally, this can be interpreted as the inability of any kind of Floer theory to produce quasi-morphisms on Liouville manifolds. By a Liouville manifold here is meant an exact symplectic manifold $(W,\omega=d\lambda)$ such that the associated Liouville vector field $Y$, given by $\iota_Y\omega=\lambda$, is complete, and such that $W$ is exhausted by compact Liouville subdomains, that is there is a sequence $(W_k)_k$ of compact codimension zero submanifolds with boundary such that $W=\bigcup_k W_k$, and such that for each $k$ we have $W_k\subset\Int W_{k+1}$, and $Y$ is positively transverse to $\partial W_k$.
\begin{rem}
  \begin{itemize}
    \item Note that the conclusion of the theorem is vanishing rather than being a multiple of the Calabi homomorphism. The reason is that the latter is not stable---it grows with the volume of the support---and therefore its only possible stable multiple is zero.
    \item The result here is for quasi-morphisms on the universal cover of the Hamiltonian group, which subsumes the case of quasi-morphisms on the group itself, since, of course, any quasi-morphism on the Hamiltonian group can be lifted to the universal cover by the projection $\wt\Ham \to \Ham$. Note that some of the above constructions yield quasi-morphisms on the Hamiltonian group itself, while others apply to its universal cover.
  \end{itemize}
\end{rem}

\noindent\tb{Discussion.} We wish to emphasize that there are examples of open symplectic manifolds of infinite volume whose Hamiltonian group admits a nonzero stable homogeneous quasi-morphism. Namely, Lanzat's construction \cite{Lanzat_QMs_Sympl_QSs_Cx_Sympl_Mfds} yields such quasi-morphisms for $1$-point blowups of certain cotangent bundles. Also, there is a general result regarding quasi-morphisms for `large' symplectic manifolds due to Kotschick \cite{Kotschick_Stable_Length_Stable_Groups}. It applies, for instance, to symplectic manifolds which are stabilizations, meaning they have the form $W=\R^2\times W'$ for some symplectic manifold $W'$. For such $W$ one can apply the results of \cite{Kotschick_Stable_Length_Stable_Groups} to deduce that only homogeneous quasi-morphisms on the Hamiltonian group of $W$ are indeed multiples of the Calabi homomorphism. In particular, this applies to subcritical Weistein manifolds, since by \cite[Theorem 14.16]{Cieliebak_Eliashberg}, such a manifold is symplectomorphic to a stabilization. It follows that our result is new only for those Liouville manifolds which are not symplectomorphic to subcritical Weinstein manifolds.

\begin{acknow}I would like to thank Petya Pushkar' and Leonid Polterovich for listening to a preliminary version of this result, and for their interest. My thanks also go to Igor Uljarevi\'c and Filip \v Zivanovi\'c for a discussion about Liouville manifolds of finite type, which ultimately allowed me to prove the result for general Liouville manifolds.
\end{acknow}

\section{Proof}

The proof crucially relies on the concept of \emph{quasi-integrals}, which are certain, not necessarily linear, functionals on the algebra of compactly supported real-valued functions of a locally compact Hausdorff space. Quasi-integrals are the subject of Section \ref{ss:QIs}, where we define them and prove a number of fundamental results about them. In Section \ref{ss:proof} we prove the theorem.

\noindent\tb{Overview of the proof.} Given a homogeneous quasi-morphism $\mu \fc \wt\Ham(W,\omega)\to\R$, we can define a functional on $C^\infty_c(W)$ via $H\mapsto \mu(\wt\phi_H)$. Thanks to the stability of $\mu$, this functional is $1$-Lipschitz with respect to the uniform norm, and thus admits a unique extension $\zeta \fc C_c(W)\to\R$, where $C_c(W)$ is the space of compactly supported continuous functions, which turns out to be a quasi-integral, and in particular it is monotone. Moreover, since $\mu$ is homogeneous, and thus conjugation-invariant, $\zeta$ is invariant under the natural action of $\Ham(W,\omega)$. The key point here is that the space of $1$-Lipschitz $\Ham(W,\omega)$-invariant quasi-integrals is \emph{compact}. Since $W$ is a Liouville manifold, the Liouville vector field $Y$ is complete, and its flow $\Delta_s$ induces, via a natural action, a $1$-parameter family of quasi-integrals $\zeta_s$ obtained from $\zeta$. It turns out that the $\zeta_s$ are also $1$-Lipschitz and $\Ham$-invariant, and therefore the family $(\zeta_s)_{s\geq 0}$ has a limit point $\eta$ as $s\to\infty$, which is also a $1$-Lipschitz $\Ham$-invariant quasi-integral. The functionals $\zeta_s$ are in general nonlinear, however, as it turns out, their nonlinearity is controlled:
$$|\zeta_s(G+H) - \zeta_s(G) - \zeta_s(H)| \leq \sqrt{2e^{-s}D(\mu)\|\{G,H\}\|_{C^0}}\,.$$
It follows that $\eta$, being a limit point of the $\zeta_s$ as $s\to\infty$, is \emph{linear}. Linear quasi-integrals are given, thanks to the Riesz representation theorem, by integration against a Borel measure. Since $\eta$ is $\Ham$-invariant, the measure representing $\eta$ must be a multiple of the volume. The only such multiple, which is also $1$-Lipschitz with respect to the uniform norm, is zero. Thus the only limit point of $(\zeta_s)_{s\geq 0}$ is the zero functional. We can then combine this with the monotonicity property of the $\zeta_s$ to deduce that $\zeta$ vanishes on a sufficiently rich family of functions, which suffices in order to conclude, again using stability, that $\mu\equiv 0$.

\begin{rem}
  Quasi-integrals arising from homogeneous quasi-morphisms as described above possess an additional property---they are \emph{symplectic} (see \cite{Monzner_Zapolsky_Comparison_Sympl_Homogenization_Calabi_QSs}), that is they are linear on Poisson-commuting functions. It is tempting to surmise that there are no symplectic quasi-integrals on Liouville manifolds, however it is futile, since they do exist, at least on cotangent bundles of symmetric Zoll spaces. Such quasi-integrals are given by the construction in \cite{Monzner_Vichery_Zapolsky_Partial_QMs_QSs_Cot_Bundles}. There is no contradiction to our main result, of course, since these quasi-integrals are not induced from quasi-morphisms on the Hamiltonian group.
\end{rem}

Let us now describe all of this in detail.

\subsection{Quasi-integrals}\label{ss:QIs}

Given a locally compact Hausdorff space $X$, let $C(X)$ be the space of real-valued continuous functions on $X$ and let $C_c(X)$ be the subspace of functions with compact support. The uniform norm on $C_c(X)$ is given by $\|f\|_{C^0} = \max_{x \in X}|f(x)|$. For $f \in C_c(X)$ we let $C(f)\subset C_c(X)$ be the smallest closed subalgebra containing $f$. Explicitly, $C(f) = \{\phi\circ f\,|\,\phi \in C(\R)\,,\phi(0)=0\}$.

A quasi-integral on $X$ is a functional $\eta \fc C_c(X) \to \R$ satisfying
\begin{itemize}
  \item \tb{(positivity):} For all $f \in C_c(X)$, $f\geq 0\Rightarrow\eta(f) \geq 0$;
  \item \tb{(quasi-linearity):} For all $f \in C_c(X)$, $\eta$ is linear on $C(f)$.
\end{itemize}
Quasi-integrals have been introduced in this generality in Rustad \cite{Rustad_Unbounded_QIs}, and independently rediscovered in \cite{Monzner_Zapolsky_Comparison_Sympl_Homogenization_Calabi_QSs}. They are a generalization of Aarnes's notion of quasi-states \cite{Aarnes_QS_QM}, which are a special case for compact spaces. Rustad showed in \cite{Rustad_Unbounded_QIs} that for each compact $K\subset X$ there is $C_K\geq 0$ so that $\eta$ is $C_K$-Lipschitz with respect to the uniform norm on the subspace of functions with support in $K$, namely if $f,g\in C_c(X)$ are supported in $K$, then $|\eta(f) - \eta(g)| \leq C_K\|f-g\|_{C^0}$. If a common value of $C_K$ exists for all compact $K$, then $\eta$ is globally Lipschitz. We will only deal with $1$-Lipschitz quasi-integrals, as those are the ones arising from stable quasi-morphisms.

\begin{exam}\label{ex:linear_QIs}
  If $\sigma$ is a compact-finite Borel measure on $X$, then $\eta(f)=\int_X f\,d\sigma$ defines a globally linear quasi-integral. Conversely, the Riesz representation theorem implies that in case $X$ is $\sigma$-compact, that is a countable union of compact subsets, then for each linear quasi-integral $\eta \fc C_c(X)\to\R$ there exists a unique Radon (that is a compact-finite regular Borel) measure $\sigma$, such that for all $f \in C_c(X)$ we have $\eta(f) = \int_Xf\,d\sigma$. In this case $\eta$ is $C$-Lipschitz for some $C\geq 0$ if and only if $\sigma(X) \leq C$.
\end{exam}

We will now formulate and prove a number of properties of quasi-integrals which will be used in the proof of Theorem \ref{thm:main}.
\begin{notation}
  The collection of $1$-Lipschitz quasi-integrals on $X$ will be denoted by $\cQ\cI_1(X)$.
\end{notation}

For the rest of the paper, we endow the collection of quasi-integrals on $X$, as well as all its subspaces, with the topology of pointwise convergence, whereby a net $\eta_\nu$ of quasi-integrals converges to $\eta$ if for all $f \in C_c(X)$ we have $\eta_\nu(f) \to \eta(f)$. We have the following fundamental property.
\begin{prop}\label{prop:compactness}
  The space $\cQ\cI_1(X)$ is compact.
\end{prop}
\noindent This result is clearly known to the experts. However, extracting it in this exact form from the existing literature is somewhat challenging, so we elected to present a self-contained proof, which appears at the end of this section. The main idea is a Banach--Alaoglu argument.

For the next property, note that homeomorphisms of $X$ naturally act on $C_c(X)$, and thus on quasi-integrals by precomposition. More precisely, if $g$ is a homeomorphism of $X$ and $\eta$ is a quasi-integral on $X$, then $g_*\eta \fc C_c(X) \to \R$ defined by $(g_*\eta)(f) = \eta(f\circ g^{-1})$ is likewise a quasi-integral. It is easy to see that the Lipschitz property with respect to the uniform norm is preserved by this action, as is the Lipschitz constant. It follows that homeomorphisms act on $\cQ\cI_1(X)$. We then have
\begin{coroll}
  If $G$ is a subgroup of the group of homeomorphisms of $X$, then the subspace of $\cQ\cI_1(X)$ consisting of $G$-invariant quasi-integrals is closed and thus likewise compact.
\end{coroll}
\begin{proof}
  It is easy to see that for $g\in G$, the subspace $\{\eta \in \cQ\cI_1(X)\,|\,g_*\eta=\eta\}$ is closed. The result now follows since the intersection of all these subspaces as $g$ ranges over $G$ is precisely the subspace of $G$-invariant quasi-integrals, which is then closed.
\end{proof}

\begin{defin}
  If $(M,\omega)$ is a symplectic manifold, a quasi-integral on $M$ is called $\Ham$-invariant if it is invariant under $\Ham(M,\omega)$.
\end{defin}

\begin{coroll}\label{cor:Ham_invt_QIs_compact}
  If $(M,\omega)$ is a symplectic manifold, then the collection of $\Ham$-invariant $1$-Lipschitz quasi-integrals on $M$ is compact. \qed
\end{coroll}

We will need the following properties of quasi-integrals in the context of quasi-morphisms on the Hamiltonian group.
\begin{prop}\label{prop:QIs_from_QMs}
  If $(W,\omega)$ is an open connected symplectic manifold, and $\mu \fc \wt\Ham_c(W,\omega) \to \R$ is a stable homogeneous quasi-morphism, then the functional $f \mapsto \mu(\wt\phi_f)$ has a unique extension to a $\Ham$-invariant $1$-Lipschitz quasi-integral $\zeta \fc C_c(W) \to \R$, which moreover satisfies the inequality
    \begin{equation}\label{eq:main_ineq}
      |\zeta(f+g) - \zeta(f) - \zeta(g)| \leq \sqrt{2D(\mu)\|\{f,g\}\|_{C^0}}\quad\text{for }f,g \in C^\infty_c(W)\,.
    \end{equation}
\end{prop}
\noindent This is a straightfoward adaptation of \cite[Section 2.5.4]{Zapolsky_PhD}, however, for the sake of completeness and ease of reading, we will present an abridged proof here.
\begin{proof}
  First, let $\zeta \fc C^\infty_c(W) \to \R$ be defined by $\zeta(f) = \mu(\wt\phi_f)$. The stability of $\mu$ implies that $\zeta$ is $1$-Lipschitz with respect to the uniform norm, and thus, by a standard argument, admits a unique $1$-Lipschitz extension $\zeta \fc C_c(W) \to \R$. Again, by stability, this extension is positive, that is $\zeta(f) \geq 0$ if $f \geq 0$. The $\Ham$-invariance is a straightforward consequence of the conjugation invariance of homogeneous quasi-morphisms. The inequality \eqref{eq:main_ineq} is proved in \cite{EPZ_QMs_Poisson_Bracket} and we will not reproduce it here. There it is proved for compact manifolds, but the proof carries over \emph{verbatim} to compactly supported Hamiltonians on an arbitrary symplectic manifold.

  It remains to prove that $\zeta$ is quasi-linear. This is done by an approximation argument. For $f \in C^\infty_c(W)$, the space $\{\phi\circ f\,|\,\phi\in C^\infty(\R)\,,\phi(0)\}$ consists of pairwise Poisson-commuting functions, and thus $\zeta$ is linear on it, thanks to the inequality. If $f \in C_c(W)$, let $f_n \in C_c^\infty(W)$ be a sequence of smooth functions with support contained in a fixed compact set, such that $\|f_n-f\|_{C^0} \to 0$. For any $g,h\in C(f)$ let $\phi,\psi\in C(\R)$ be such that $\phi(0) = \psi(0)=0$ and $g=\phi\circ f$, $h=\psi\circ f$. Let $I\subset \R$ be a finite closed interval containing $\im f$ and $\im f_n$ for all $n$. Let $\phi_n,\psi_n \in C^\infty(\R)$ satisfy $\phi_n(0) = \psi_n(0) = 0$ and $\|\phi_n - \phi\|_{C^0(I)}, \|\psi_n-\psi\|_{C^0(I)} \to 0$. It is easy to see that $\phi_n\circ f_n \to \phi\circ f = g$, $\psi_n\circ f_n \to \psi\circ f = h$ in the $C^0$-norm. Since $\{\phi_n\circ f_n,\psi_n\circ f_n\}=0$ for all $n$, we have $\zeta(\phi_n\circ f_n + \psi_n\circ f_n) = \zeta(\phi_n\circ f_n) + \zeta(\psi_n\circ f_n)$ for all $n$. Passing to the limit and using the Lipschitz property of $\zeta$, we deduce that $\zeta(g+h) = \zeta(g) + \zeta(h)$. Thus $\zeta$ is additive on $C(f)$. The homogeneity is proved similarly. We conclude that $\zeta$ is linear on $C(f)$ and the proof is complete.

\end{proof}

The following lemma is straightforward and is left to the reader as an exercise.
\begin{lemma}\label{lem:homogeneity}
  If $(M,\omega)$ is a symplectic manifold and $\mu \fc \wt\Ham_c(W,\omega) \to \R$ is a stable homogeneous quasi-morphism, then for all $H \in C^\infty_c(W)$ and $s\in\R$ we have $\mu(\wt\phi_{sH}) = s\mu(\wt\phi_H)$. \qed
\end{lemma}

Lastly, we will use the following measure-theoretic result in the context of symplectic geometry. It is undoubtedly known to experts, but we were unable to find a reference in the literature, and thus an independent proof is presented here.
\begin{lemma}\label{lem:Ham_invt_measures}
  Let $(M,\omega)$ be a connected symplectic manifold, and let $\sigma$ be a compact-finite Borel measure on $M$, and which is invariant under the action of $\Ham(M,\omega)$. Then there is $\kappa\geq 0$ such that $\sigma = \kappa\Vol$.
\end{lemma}
\noindent Here $\Vol$ stands for the Lebesgue measure induced by the volume form $\omega^{\wedge\frac12\dim M}$.
\begin{proof}
  We will prove the lemma in a sequence of steps. Throughout, by a cube in $\R^{2n}$ we mean a cube with sides parallel to the coordinate axes.

  \tb{Step 1:} The case of $M=(0,1)^{2n}\subset \R^{2n}$ with the restriction of the standard symplectic form. If $K\subset M$ is any compact set contained in a hyperplane, then $\sigma(K)=0$. This is because we can produce an uncountable collection of pairwise disjoint Hamiltonian translates of $K$, all of which have the same measure as $K$, forcing it to vanish. In particular, if $C\subset M$ is any cube, then $\sigma(\partial C) = 0$, whence $C$ and its interior have the same measure. Likewise, any linear translate of $C$ contained in $M$ can be obtained from $C$ by applying to it a Hamiltonian diffeomorphism of $M$, and thus it has the same measure as $C$. It follows that if we subdivide a closed cube $C\subset M$ into $N^{2n}$ pairwise congruent subcubes $(C_i)_{i=1}^{N^{2n}}$, $N\in \N$, then, since all the $C_i$ are translates of one another, they all have the same measure, and thus $\sigma(C)=N^{2n}\sigma(C_1)$, since the boundaries have measure zero.

  It follows that if $C,C'\subset M$ are cubes with rational sides, then
  $$\frac{\sigma(C)}{\sigma(C')} = \frac{\Vol(C)}{\Vol(C')}\,.$$
  Indeed, under the assumptions we can find a closed cube $D\subset M$ such that both $C$ and $C'$ are unions of translates of $D$ with pairwise disjoint interiors. It follows that there is $\kappa \geq 0$ such that $\sigma=\kappa\Vol$ on all cubes with rational sides.

  If $K\subset M$ is any compact set, then there exists a decreasing sequence $R_k$ of sets which are finite unions of closed cubes with rational sides and pairwise disjoint interiors, such that $K = \bigcap_kR_k$. It follows from \cite[Chapter 6, Proposition 1.3]{Stein_Shakarchi_Real_analysis_3} that a finite Borel measure on a compact metrizable space is regular. Since $\sigma$ is by assumption compact-finite, its restriction to $R_1$ is a finite Borel measure, and therefore it is regular, since, of course, $R_1$ is compact and metrizable. It follows that
  $$\sigma(K) = \lim_k\sigma(R_k) = \kappa\lim_k\Vol(R_k) = \kappa\Vol(K)\,.$$
  A regular Borel measure is determined by its values on compacts, and the first step is complete.

  \tb{Step 2:} $M \subset \R^{2n}$ is any cube. This follows from Step 1 by rescaling.

  \tb{Step 3:} $M\subset \R^{2n}$ is any connected open set. From the previous step it follows that if $C\subset M$ is any open cube, then there is $\kappa_C\geq 0$ such that $\sigma|_C = \kappa_C\Vol$. If $C'\subset M$ is another straight cube, then, of course, there is $\kappa_{C'}$ with $\sigma|_{C'} = \kappa_{C'}\Vol$. Since $M$ is connected, there is $\phi\in\Ham_c(M,\omega)$ with $\phi(C')\cap C\neq \varnothing$. Let $C''\subset C\cap\phi(C')$ be any straight cube. Thanks to the invariance of $\sigma$, we have $\sigma|_{\phi(C')} = \kappa_{C'}\Vol$. We have:
  $$\kappa_C\Vol|_{C''} = \big(\kappa_C\Vol|_C)|_{C''} = \big(\sigma|_C\big)|_{C''} = \sigma|_{C''} = \big(\sigma|_{\phi(C')}\big)|_{C''} = \big(\kappa_{C'}\Vol|_{\phi(C')}\big)|_{C''} = \kappa_{C'}\Vol|_{C''}\,,$$
  whence $\kappa_C=\kappa_{C'}$. Let us denote by $\kappa$ the common value of the constants $\kappa_C$ for all cubes $C\subset M$. Now there is a countable collection of closed cubes $C_i$ with pairwise disjoint interiors such that $M = \bigcup_iC_i$. It follows that for any Borel set $B\subset M$ we have
  $$\sigma(B) = \sum_i\sigma(B\cap C_i) = \sum_i\sigma|_{C_i}(B\cap C_i) = \sum_i\kappa\Vol(B\cap C_i) = \kappa\Vol(B)\,.$$

  \tb{Step 4:} There is $\kappa \geq 0$ such that $\sigma=\kappa\Vol$ on open connected subsets of $M$ contained in Darboux balls. Indeed, if $U\subset M$ is open, connected, and contained in a Darboux ball, let $\psi \fc U \to \R^{2n}$ be a symplectic embedding, and let $\tau=\psi_*\sigma$. Then $\tau$ is a Borel compact-finite measure on $\psi(U)$ which is invariant under Hamiltonian isotopies with compact supports in $\psi(U)$. By the previous step, there is $\kappa_U$ such that $\tau = \kappa_U\Vol$, whence $\sigma|_U = \kappa_U\Vol$. If $V\subset M$ is another such open set, we have $\sigma|_V=\kappa_V\Vol$. Using the same trick we had in the previous step for cubes, we obtain $\kappa_U = \kappa_V$.

  \tb{Step 5:} The general case. Let $\kappa$ be the constant from the previous step. It suffices to show that $\sigma = \kappa\Vol$ on compact subsets. If $K\subset M$ is compact, let $U_1,\dots,U_k\subset M$ be a collection of connected open sets, each one contained in a Darboux ball, such that $K\subset\bigcup_{i=1}^kU_i$. Then
  $$K = \bigsqcup_{i=1}^k\left(K\cap\bigg(U_i\setminus\bigcup_{j=1}^{i-1}U_j\bigg)\right)\,,$$
  whence
  \begin{multline*}
    \sigma(K) = \sum_{i=1}^{k}\sigma\left(K\cap\bigg(U_i\setminus\bigcup_{j=1}^{i-1}U_j\bigg)\right) = \sum_{i=1}^{k}\sigma|_{U_i}\left(K\cap\bigg(U_i\setminus\bigcup_{j=1}^{i-1}U_j\bigg)\right)
     \\ = \kappa\sum_{i=1}^{k}\Vol\left(K\cap\bigg(U_i\setminus\bigcup_{j=1}^{i-1}U_j\bigg)\right) = \kappa\Vol(K)\,.
  \end{multline*}
  The proof of the lemma is complete.
\end{proof}

We close this section with a proof of Proposition \ref{prop:compactness}. Before doing this, let us mention two ways the result can be extracted from the literature. Firstly, one can prove, based on \cite{Rustad_Unbounded_QIs}, that the space of globally Lipschitz quasi-integrals on $X$ is homeomorphic to the space of quasi-integrals on the one-point compactification of $X$. The compactness of the latter space has been known essentially from the foundations of the theory. See \cite{Aarnes_Pure_QS_Extremal_QM}, for instance. Another way is to use Grubb's compactness results for topological measures on locally compact spaces (there called quasi-measures), and then to prove that the representation theorem of Rustad relating topological measures and quasi-integrals in fact yields a homeomorphism between the two spaces. None of these ways is particularly difficult, but it would require the introduction of unnecessary additional concepts, and we doubt it would be simpler than the proof presented here.

\begin{proof}
  The basic idea is to use the Banach--Alaoglu theorem about the compactness of the unit ball in the dual space of a normed vector space with respect to the weak-* topology and the Tychonoff theorem about the compactness of products.

  For $f \in C_c(X)$, $C(f)$ is a Banach space with respect to the uniform norm, which is in fact isomorphic to the Banach space $C_0(\im f)$ of continuous functions on $\im f$ vanishing at zero, again with the uniform norm. Consider its continuous dual $C(f)'$ with the dual norm. The corresponding closed unit ball $B_1(f)\subset C(f)'$ is compact with respect to the weak-* topology, thanks to the Banach--Alaoglu theorem. The subspace of positive functionals
  $$C(f)'_+ = \{\rho\in C(f)'\,|\,\forall h\in C(f):h\geq 0\Rightarrow \rho(h) \geq 0\}\subset C(f)'$$
  is closed in the weak-* topology, and thus the positive part $B_1(f)_+ = B_1(f)\cap C(f)'_+$ of $B_1(f)$ is closed in it, and therefore compact in the weak-* topology.

  Consider the map
  $$\iota \fc \cQ\cI_1(X) \to \prod_{f\in C_c(X)}B_1(f)_+\,,\quad \iota(\eta) = \big(\eta|_{C(f)}\big)_f\,,$$
  where the product is endowed with the product topology, while each factor is given the weak-* topology. The factors being compact, the product is itself compact, thanks to the Tychonoff theorem. Therefore in order to prove that $\cQ\cI_1(X)$ is compact, it suffices to show that $\iota$ is a closed embedding, meaning $\im \iota$ is closed and $\iota$ is a homeomorphism onto $\im \iota$. This is accomplished in a series of steps.

  \tb{Step 1: $\iota$ is injective.} Let $\eta,\eta' \in \cQ\cI_1(X)$ satisfy $\eta|_{C(f)}=\eta'|_{C(f)}$ for all $f$; then
  $$\eta(f) = \big(\eta|_{C(f)}\big)(f) = \big(\eta'|_{C(f)}\big)(f) = \eta'(f)\,,$$
  whence $\eta = \eta'$.

  \tb{Step 2: $\iota$ is continuous.} Let $\eta_\nu \in \cQ\cI_1(X)$ be a net converging to $\eta \in \cQ\cI_1(X)$, meaning $\eta_\nu(f) \to\eta(f)$ for each $f \in C_c(X)$. For a fixed $f\in C_c(X)$, the restrictions $\eta_\nu|_{C(f)}$ then clearly converge to $\eta|_{C(f)}$ in the weak-* topology, which is just another name for the topology of pointwise convergence. Thus $\iota(\eta_\nu)$ converges to $\iota(\eta)$ in each coordinate, or, in other words, in the product topology, as required.

  \tb{Step 3: $\im\iota$ is compact.} It suffices to prove that it is closed, since the target space is already compact. Since the coordinates $\iota(\eta)_f=\eta|_{C(f)}$ of $\iota(\eta)$ are given by restrictions of the same functional, they satisfy the obvious coherence relations $\iota(\eta)_f|_{C(g)} = \iota(\eta)_g$. Let us therefore define the corresponding subspace
  $$\cZ = \bigg\{(\rho_h)_h \in \prod_{h\in C_c(X)}B_1(h)_+\,\Big|\,\forall f,g\in C_c(X):g\in C(f)\Rightarrow \rho_f|_{C(g)} = \rho_g\bigg\}\,.$$
  Clearly $\im\iota \subset \cZ$. We will show that $\cZ = \im\iota$ and that it is closed.

  Each individual coherence relation can be encoded in a suitable subspace, as follows. For $f,g\in C_c(X)$ such that $g \in C(f)$ consider
  $$\cZ_{f,g} = \bigg\{(\rho_h)_h \in \prod_{h\in C_c(X)}B_1(h)_+\,\Big|\,\rho_f|_{C(g)} = \rho_g\bigg\}\,.$$
  It is then obvious that
  $$\cZ = \bigcap_{f,g\in C_c(X):\,g\in C(f)}\cZ_{f,g}\,.$$
  To prove $\cZ$ is closed, it suffices to show that each $\cZ_{f,g}$ is. To see this, consider the maps
  $$\pr_g\,,\res_g \fc \prod_{h\in C_c(X)}B_1(h)_+ \to B_1(g)_+\,,\quad\text{defined by}$$
  $$\pr_g\big((\rho_h)_h\big)=\rho_g\text{ and }\res_g\big((\rho_h)_h\big) = \rho_f|_{C(g)}\,.$$
  The map $\pr_g$ is continuous thanks to the definition of the product topology. The map $\res_g$ is the composition of the continuous map $\pr_f$ with the restriction map $C(f)'\to C(g)'$ dual to the inclusion $C(g)\to C(f)$. The restriction map is, of course, continuous in the weak-* topology, as it is simply the topology of pointwise convergence. It follows that $\res_g$ is likewise continuous. Since the weak-* topology is Hausdorff, the diagonal $\Delta_{B_1(g)_+}\subset B_1(g)_+\times B_1(g)_+$ is closed, and thus so is its continuous preimage
  $$\cZ_{f,g} = \big(\pr_g,\res_g\big)^{-1}(\Delta_{B_1(g)_+})\,,\quad\text{where }\big(\pr_g,\res_g\big) \fc \prod_{h\in C_c(X)}B_1(h)_+ \to B_1(g)_+\times B_1(g)_+\,.$$
  It remains to show that $\cZ\subset\im\iota$. Let $(\rho_h)_h \in \cZ$ and define $\eta \fc C_c(X) \to \R$ by $\eta(f) = \rho_f(f)$. We must show that $\eta$ is indeed a $1$-Lipschitz quasi-integral, and that $\iota(\eta) = (\rho_h)_h$. To show $\eta$ is a quasi-integral, we need to prove its positivity and quasi-linearity. For positivity, if $f\geq 0$, then, since $\rho_f \in B_1(f)_+$, we have $\eta(f)=\rho_f(f) \geq 0$. For quasi-linearity, if $f \in C_c(X)$, we mush show that $\eta|_{C(f)}$ is linear. For this, it suffices to prove that $\eta|_{C(f)} = \rho_f$. Let $g \in C(f)$. Then by coherence $\eta(g) = \rho_g(g) = \big(\rho_f|_{C(g)}\big)(g) = \rho_f(g)$, as claimed. The $1$-Lipschitz property of $\eta$ follows from the fact that each $\rho_f$ has norm at most one, and thus is itself $1$-Lipschitz with respect to the uniform norm on $C(f)$. The fact that $\iota(\eta)=(\rho_h)_h$ is immediate from the definition of $\eta$.

  \tb{Step 4: $\iota$ is a homeomorphism onto its image.} It suffices to prove that if $\eta_\nu\in\cQ\cI_1(X)$ is a net, $\eta\in\cQ\cI_1(X)$, and $\iota(\eta_\nu) \to \iota(\eta)$, then $\eta_\nu\to \eta$. Fix $f \in C_c(X)$. Since $\iota(\eta_\nu) \to \iota(\eta)$, from the definition of product topology it follows that $\eta_\nu|_{C(f)} = \iota(\eta_\nu)_f  \to \iota(\eta)_f = \eta|_{C(f)}$ in the weak-* topology. In particular
  $$\eta_\nu(f) = \big(\eta_\nu|_{C(f)}\big)(f) \to \big(\eta|_{C(f)}\big)(f) = \eta(f)\,,$$
  as claimed. This finishes the proof.
\end{proof}

\subsection{Proof of Theorem \ref{thm:main}}\label{ss:proof}
  Let $\lambda$ be the Liouville form on $W$ and $Y$ the associated Liouville vector field given by $\iota_Y\omega=\lambda$. For $s\in \R$ let $\Delta_s$ denote the time-$s$ flow of $Y$. For $\phi \in \Ham(W,\omega)$ define
  $$\phi^{\Delta_s}:=\Delta_s\circ\phi\circ\Delta_s^{-1}\,.$$
  We have $\phi^{\Delta_s} \in \Ham(W,\omega)$, and in fact if $H$ is a time-dependent Hamiltonian on $W$, then $\phi_H^{\Delta_s}=\phi_{e^s\,H\circ\Delta_s^{-1}}$. This conjugation action naturally extends to Hamiltonian isotopies, and therefore to $\wt\Ham(W,\omega)$:
  $$\wt\Ham(W,\omega)\ni\wt\phi=[(\phi_t)_{t\in[0,1]}]\mapsto \wt\phi^{\Delta_s}:=\big[\big(\phi_t^{\Delta_s}\big)_{t\in[0,1]}\big] \in \wt\Ham(W,\omega)\,,$$
  where brackets denote the homotopy class of the isotopy. One can check that this action is well-defined. Moreover, each map $\wt\phi\mapsto\wt\phi^{\Delta_s}$ is an automorphism of $\wt\Ham(M,\omega)$.

  Let $\mu_s \fc \wt\Ham(W,\omega)\to\R$ be given by
  $$\mu_s(\wt\phi) = e^{-s}\mu(\wt\phi^{\Delta_s})\,.$$
  Since composing a quasi-morphism on a group with an automorphism yields a quasi-morphism with the same defect, we see that $\mu_s$ is a homogeneous quasi-morphism with defect $D(\mu_s) = e^{-s}D(\mu)$.

  Consider the functionals $\zeta,\zeta_s \fc C^\infty_c(W) \to \R$ given by
  $$\zeta_s(H) = \mu_s(\wt\phi_H)\quad\text{and}\quad \zeta=\zeta_0\,.$$
  By Proposition \ref{prop:QIs_from_QMs}, these are $1$-Lipschitz quasi-integrals on $W$. We have
  $$\zeta_s(H) = \mu_s(\wt\phi_H)=e^{-s}\mu(\wt\phi_{e^sH\circ\Delta_s^{-1}})\stackrel{*}{=}\mu(\wt\phi_{H\circ\Delta_s^{-1}})=\zeta(H\circ\Delta_s^{-1}) = (\Delta_{s*}\zeta)(H)\,,$$
  where $\stackrel{*}{=}$ follows from Lemma \ref{lem:homogeneity}. Thanks to Proposition \ref{prop:QIs_from_QMs}, all the $\zeta_s$ are $\Ham$-invariant. By Corollary \ref{cor:Ham_invt_QIs_compact}, the collection of $\Ham$-invariant $1$-Lipschitz quasi-integrals is compact. It follows that there is a subnet\footnote{We opted to use here the characterization of compactness in terms of nets. While it is true that the space of globally Lipschitz quasi-integrals is metrizable, and thus compactness can be characterized in terms of sequences, a proper discussion of this would take us too far afield. While the notion of subnets is slightly unconventional, for this paper there is no difference from subsequences in practical terms.} $s_\nu\to\infty$ and a $\Ham$-invariant $1$-Lipschitz quasi-integral $\eta$ such that $\zeta_{s_\nu} \to \eta$. By Proposition \ref{prop:QIs_from_QMs} we have, for all $G,H \in C^\infty_c(W)$:
  $$|\zeta_{s_\nu}(G+H) - \zeta_{s_\nu}(G) - \zeta_{s_\nu}(H)| \leq \sqrt{2D(\mu_{s_\nu}) \|\{G,H\}\|_{C^0}} = \sqrt{2e^{-{s_\nu}}D(\mu)\|\{G,H\}\|_{C^0}}\,.$$
  Taking the limits of both sides, we arrive at the inequality
  $$|\eta(G+H) - \eta(G) - \eta(H)| \leq 0\,,$$
  meaning $\eta$ is globally linear, that is $\eta$ is the integration functional with respect to a Radon measure $\sigma$, see Example \ref{ex:linear_QIs}. Since $\eta$ is $\Ham$-invariant, so is $\sigma$. By Lemma \ref{lem:Ham_invt_measures}, the only $\Ham$-invariant compact-finite Borel measures on $W$ are multiples of volume. Therefore $\sigma$, being a $\Ham$-invariant Radon measure, thus in particular a compact-finite Borel measure, is a multiple of volume. However, $\eta$ is also $1$-Lipschitz, which means that the total mass of $\sigma$ is at most $1$, again see Example \ref{ex:linear_QIs}. Since $W$ is a complete Liouville manifold, it has infinite volume, and thus the only possible multiple of volume for $\sigma$ is zero. We conclude that $\eta = 0$.

  We will now show that $\eta=0$ implies that $\mu=0$. Our definition of Liouville manifolds is that $W$ is exhausted by a sequence of Liouville domains $W_k$, that is each $W_k$ is a codimension zero compact connected submanifold with boundary such that $Y$ is everywhere positively transverse to $\partial W_k$, such that $W_k\subset \Int W_{k+1}$ for each $k$, and such that $W = \bigcup_k W_k$. Let us fix such a sequence of domains $W_k$. For each $k$ let us define
  $$\wh W_k = \bigcup_{s\geq 0}\Delta_s(W_k)\,.$$
  This is an open submanifold of $W$, which is symplectomorphic to the completion of the Liouville domain $W_k$. The Liouville flow allows us to embed the manifold $\R\times \partial W_k$ into $\wh W_k$, and therefore into $W$, namely we have the embedding $\R\times \partial W_k \to W$, $(s,w)\mapsto \Delta_s(w)$.

  Fix now a smooth nonincreasing function $\chi \fc \R \to \R$ such that $\chi|_{(-\infty,0]}\equiv 1$ and $\chi|_{[1,\infty)}\equiv 1$. We will construct from it a family of functions $\chi_{k,s} \in C^\infty_c(W)$. First, define $\chi_k$ as follows. Consider the above diffeomorphism $\R\times\partial W_k \to \wh W_k\setminus \Sigma_k$, $(s,w)\mapsto \Delta_s(w)$, where $\Sigma_k = \bigcap_{s\leq 0}\Delta_s(W_k)$ is the Liouville skeleton of $W_k$. On the image $\wh W_k\setminus\Sigma_k$ of this diffeomorphism, $\chi_k$ is defined as the composition of the inverse of the diffeomorphism with the projection $\R\times \partial W_k \to \R$. On $\Sigma_k$, $\chi_k\equiv 1$, and on the rest of $W$, $\chi_k\equiv 0$. It is easy to see that $\chi_k$ is indeed a smooth function with compact support. For $s\in \R$ set $\chi_{k,s}=\chi_k\circ \Delta_s^{-1}$.

  By construction, for fixed $k$, $(\chi_{k,s})_{s\in\R}$ is a nondecreasing family of functions. It follows that $\zeta(\chi_{k,s})$ is a nondecreasing function of $s$. We have
  $$\zeta(\chi_{k,s}) = \zeta(\chi_k\circ\Delta_s^{-1}) = \zeta_s(\chi_k)\,.$$
  Recall the net $s_\nu \to \infty$ for which $\zeta_{s_\nu}\to\eta = 0$. In particular we have $\zeta_{s_\nu}(\chi_k) \to \eta(\chi_k) = 0$. We claim that this implies that $\zeta(\chi_k) = 0$. Indeed, let $\epsilon > 0$. From the definition of convergence of nets, there is $\nu_0$ such that for all $\nu\succeq\nu_0$ we have $|\zeta_{s_\nu}(\chi_k)| \leq \epsilon$. Since $s_\nu\to \infty$, we can in particular choose $\nu_0$ so that $\nu\succeq \nu_0$ implies $s_\nu \geq 0$. It then follows that for all such $\nu$ we have
  $$\zeta(\chi_k) = \zeta_0(\chi_k) \stackrel *\leq \zeta_{s_\nu}(\chi_k) \leq \epsilon\,,$$
  where $\stackrel*\leq$ is because $\zeta_s(\chi_k)$ is a nondecreasing function of $s$ and $s_\nu \geq 0$ for $\nu\succeq \nu_0$. Since $\epsilon$ was chosen arbitrary, this implies that indeed we have $\zeta(\chi_k) = 0$.

  Now, if $H$ is any compactly supported time-dependent Hamiltonian on $W$, its support is contained in some $W_k$, and, taking $b=\max_{W\times[0,1]}H$ and $a=\min_{W\times[0,1]}H$, we have $a\chi_k \leq H_t \leq b\chi_k$ for all $t$, whence by the stability of $\mu$:
  $$0=\zeta(a\chi_k) = \mu(\wt\phi_{a\chi_k}) \leq \mu(\wt\phi_H) \leq \mu(\wt\phi_{b\chi_k}) = \zeta(b\chi_k) = 0\,,$$
  implying $\mu(\wt\phi_H) = 0$. The proof is complete.

\bibliographystyle{alpha}
\bibliography{bibfile}

\noindent
\begin{tabular}{l}
University of Haifa \\
Department of Mathematics \\
Faculty of Natural Sciences \\
3498838, Haifa, Israel \\
\&\\
MI SANU \\
Kneza Mihaila 36 \\
Belgrade 11001\\
Serbia\\
{\em E-mail:}  \texttt{frol.zapolsky@gmail.com}
\end{tabular}

\end{document}